\documentclass{elsart3-1}

\usepackage{amssymb,amsmath}

\usepackage[english]{babel}

\newtheorem{theorem}{Theorem}[section]
\newtheorem{lemma}[theorem]{Lemma}
\newtheorem{e-proposition}[theorem]{Proposition}
\newtheorem{corollary}[theorem]{Corollary}
\newtheorem{definition}[theorem]{Definition\rm}

\newenvironment{proof}{{\bf Proof}:}{\hfill$\square$}

\renewcommand{\theequation}{\arabic{equation}}
\newcommand{\rad}{\rm{rad}}

\def\og{\leavevmode\raise.3ex\hbox{$\scriptscriptstyle\langle\!\langle$~}}
\def\fg{\leavevmode\raise.3ex\hbox{~$\!\scriptscriptstyle\,\rangle\!\rangle$}}

\begin{document}
\centerline{}
\begin{frontmatter}


\title{On numbers satisfying Robin's inequality, properties of the next counterexample and improved specific bounds}
\author{Robert VOJAK}
\ead{vojakrob@gmail.com}
\address{Croatia \& France}
\selectlanguage{english}


\selectlanguage{english}


\medskip
\begin{abstract}
\selectlanguage{english}
Define $s (n) := n^{- 1} \sigma (n)$ ($\sigma (n):=\sum_{d|n}d )$ and $\omega(n)$ is the number of prime divisors of $n$. One of the properties of $s$ plays a central role: $s (p^a) > s (q^b)$ if $p <
q$ are prime numbers, with no special condition on $a, b$ other than $a, b
\geqslant 1$. This result, combined with the Multiplicity Permutation theorem,
will help us establish  properties of the next counterexample (say
$c$) to Robin's inequality $s (n) < e^{\gamma} \log \log n$. The number $c$ is superabundant, and  $\omega(c)$  must be greater than a number close to one billion. In addition, the
ratio $p_{\omega (c)} / \log c$ has a lower and upper bound. At most $\omega(c)/14$ multiplicity parameters are greater than $1$. Last but not least, we apply simple methods to sharpen Robin's inequality for
various categories of numbers.\\
\end{abstract}
\end{frontmatter}


\setcounter{section}{0}
\section{Introduction}\label{intro}
The following notations will be used throughout the article:\\	
\begin{itemize}
\item[$\bullet$ ] $\gamma:$ Euler's constant,
\item[$\bullet$ ] $\lfloor x \rfloor$ denotes the largest integer not exceeding $x$
\item[$\bullet$ ] $p_i$: $i$-th prime number,
\item[$\bullet$ ] $\sigma (n) = \sum_{d | n } d$,
\item[$\bullet$ ] $\varphi(n)$  counts the positive integers up to a given integer $n$ that are relatively prime to $n$
\item[$\bullet$ ] $\omega(n)$ is  the number of distinct primes dividing $n$,
\item[$\bullet$ ] $N_i = \prod_{j \leqslant i} p_j$ (primorial),
\item[$\bullet$ ] $\rad (n) = \underset{p\,\text{prime}}{\underset{p | n}{\prod}} p$, radical of $n$.
 \end{itemize}
\ \\ \\
Define $s (n) = \dfrac{\sigma (n)}{n}$, $f (n) =\dfrac{n}{\varphi (n)}$. 
\ \\ \\
\ \\
A number $n$ satisfies {\it Robin's inequality}\cite{robin} if and only if $n \geqslant 5041$, and 
$$s(n) < e^\gamma \log \log n$$
\  \\
\begin{definition}
A number $n$ is {\it superabundant} (SA) \cite{erdos} if for all $m<n$, we have $s(m)<s(n)$.\\
\end{definition}
\begin{definition}
A number $n$ is a {\it Hardy-Ramanujan} number (HR) \cite{cho} if $n=\prod\limits_{i \leqslant \omega(n)}p_i^{a_i}$ with $a_i \geqslant a_{i+1} \geqslant 1$ for all $i \leqslant \omega(n)-1$.\\
\end{definition}
All the results presented hereafter concern Robin's inequality, which is known to be equivalent to  Riemann's Hypothesis\cite{robin}. We start by establishing an important result about the multiplicities in the prime factorization of a number $n$, and how switching  these multiplicities affects $n$ and $s(n)$. \\
Next, we give some categories of numbers satisfying Robin's inequality, and for those who do not, we exhibit some of their properties (theorem \ref{counter}).\\
The last part is devoted to sharpening  Robin's inequality in specific cases.\\ \\
In the sequel, the notation $q_1, q_2, ...$ will be used to define a finite strictly increasing sequence of prime numbers. The following theorems represent our main results (the proofs will be given later on):\\
\begin{theorem}\label{halfpi}
  Let $n = \underset{i=1}{\overset{\omega(n)}{\prod}}  q^{a_i}_i$ be a positive integer.
 If 
\begin{itemize}
\item $\omega (n) \geqslant 18$ and $\# \left\{ i \leqslant \omega (n) \: ; \:
  a_i \neq 1 \right\} \geqslant \frac{\omega (n) }{2}$, or
\item $\omega (n) \geqslant 39$ and $\# \left\{ i \leqslant \omega (n) \: ; \:
  a_i \neq 1 \right\} \geqslant \frac{\omega (n) }{3}$, or
\item $\omega (n) \geqslant 969672728$ and $\# \left\{ i \leqslant \omega (n) \: ; \:
  a_i \neq 1 \right\} \geqslant \frac{\omega (n) }{14}$
\end{itemize}
\ \\
then $n$ satisfies Robin's inequality.
\end{theorem}
\ \\
\begin{theorem}\label{domination}
Let $n = \underset{i \leqslant k}{\prod} q^{a_i}_i$. For all positive integers $k, q$, define
\begin{equation}\label{mndef}
 M (k) = e^{e^{- \gamma} f (N_k)} - \log N_k
\end {equation}
\[ M_k (q) = 1 + \frac{M(k)}{\log q}  
\]
If for some $i \leqslant \omega (n)$, we have $a_i \geqslant M_{\omega(n)} (q_i)$, then the integer
$n$ satisfies Robin's inequality, regardless of the multiplicities of the other prime divisors.\\
\end{theorem}
For instance, if $n=\prod_{i\leqslant 5} p_i^{a_i}$, $M_5(2)= 11.3367 \ldots$ All integers of the form $2^{a_1}\prod_{i=2}^5 p_i^{a_i}$ satisfy Robin's inequality if $a_1 \geqslant 12$, and $a_2, a_3, a_4, a_5 \geqslant 1$.
\ \\
\begin{corollary}
Let $n$ be a positive integer. If $f(N_{\omega(n)})-e^\gamma \log \log N_{\omega(n)} \leqslant 0$, then $n$ satisfies Robin's inequality.	
\end{corollary}
\ \\
 \begin{proof}
If $f(N_{\omega(n)})-e^\gamma \log \log N_{\omega(n)} \leqslant 0$, we have $M(k) \leqslant 0$, and for all $i \leqslant \omega(n)$ and all prime numbers $p$, 
$$
a_i \geqslant 1 \geqslant M_{\omega(n)}(p)  
$$
We conclude using theorem \ref{domination}.
\end{proof}
\ \\
\begin{theorem}\label{counter}
Let $c = \underset{i=1}{\overset{\omega(n)}{\prod}}  q^{c_i}_i  (c \geqslant 5041)$ be the least number (if it exists) not satisfying Robin's
inequality.\\
The following properties hold:
\begin{enumerate}
\item[$1.$] $c$ is superabundant,
\item[$2.$] $\omega(c) \geqslant 969\,672\,728$,
\item[$3.$] $\#\left\{i \leqslant \omega(c)\,;\,c_i \neq 1\right\} < \dfrac{\omega(c)}{14}$,
\item[$4.$]  $e^{-\frac{1}{\log p_{\omega  (c)}}} < \dfrac{p_{\omega(c)}}{\log c} < 1$,
\item[$5.$]  for all $1< i \leqslant \omega(c) $,
$$
p_i ^{c_i}< \min\left(2^{c_1+2}, p_i\,e^{M(k)}\right)
$$
 \end{enumerate}
\end{theorem}
\section{Preliminary results}
We will need some properties about the functions $s$ and $f$ to prove the above-mentioned theorems.\\
\begin{lemma}\label{AF}
Let $n = \prod\limits_{i \leqslant k}p^{a_i}_i$, $p, q$ be two prime numbers, and $k, a$ and $b$ positive integers.
\begin{enumerate}
\item[$1.$ ] the functions $s$ and $f$ are multiplicative 
\item[$2.$ ] $s(n)=\prod\limits_{i \leqslant k}\dfrac{p_i-p^{-a_i}_i}{p_i-1}$,  $f (n) = \prod\limits_{i \leqslant k}\dfrac{p_i}{p_i-1}$ and $f(n)=f(\rad(n))$.
\item[$3.$ ] $1 + \dfrac{1}{p} \leqslant s (p^k) < \dfrac{p}{p - 1}$ and $\displaystyle \lim_{k \mapsto \infty} s(p^k)=  \frac{p}{p - 1} = f(p)$
\item[$4.$ ] \label{prop1}If $p<q$, $s (p^a) > s (q^b)$ for all positive integers $a,b$
\item[$5.$ ] $s(p^k)$ increases as $k>0$ increases, decreases as $p$ increases,
\item[$6.$\label{prop3} ]  $\dfrac{s (p^{k + 1})}{s (p^k)}$ decreases as $k$ increases,
\item[$7.$ ] \label{spa}If $a>b>0$, $\dfrac{s (p^{a})}{s (p^b)}$ decreases as $p$ incresases,
\item[$8.$ ] Let $p$ be a prime number, and $q$ any prime number, $q < p (p + 1)$. Then, for all $a, b \geqslant 1$
\[ \frac{s (p^{a + 1})}{s (p^a)} < s (q^b) \]
For instance: $s(p^{a+1}_n) <  s(p^{a}_n)\, s(p_{n+1}^b)$ for all $a,b \geqslant 1$.\\
\item[$9.$ ] If $m,n \geqslant 2$, $s(m n) \leqslant s(m) s(n)$.
\end{enumerate}
\ \\
IMPORTANT: note that $s(p^a) > s(q^b)$  requires only to have $p<q$, and no condition on positive integers $a, b$.\\
\end{lemma}
\begin{proof}
\begin{enumerate}
\item[$1.$ ] The multiplicativity of $s$ and $f$ is a consequence of the multiplicativity of $\sigma$ and $\varphi$,
\item[$2.$ ] The proof is straightforward: recall that $\rad (p^k)=p$, $\varphi (p_{}^k) = p^k - p^{k - 1}$, and hence $f (p^k) = \frac{p}{p - 1} = f (\rad (p^k))$, and use the multiplicativity of $f$ and $\rad$ to conclude.
\item[$3.$ ]  It is deduced from the properties $\sigma(p^k)=\frac{p^{k+1}-1}{p-1}$ and $\varphi(p^k)=p^{k}-p^{k-1}$.
To prove the asymptotic result, use $s(p^k)=\frac{p-\frac{1}{p^k}}{p-1}$,
\item[$4.$ ] We have $$ 1 + \frac{1}{p} \geqslant \frac{q}{q - 1} $$ when $q>p$, and we deduce  $$s (p^a) \geqslant 1 + \frac{1}{p}\geqslant \frac{q}{q - 1} > s(q^b)$$
\item[$5.$ ] Use $s (p^k) = \frac{p - \frac{1}{p^k}}{p - 1}$, and we can conclude that $s(p^k)$ increases as $k$ increases. Use (\ref{prop1}) with $a=b=k$ to prove that if $p<q$, $s(p^k) > s(q^k)$.
\item[$6.$ ] We have $$\frac{s (p^{k + 1})}{s (p^k)} = \frac{p^{k + 2} - 1}{p^{k + 2} - p}$$ which decreases as $p^k$ increases.\\ \\
Let us show that it also decreases as $p$ increases.
The sign of the first derivative of $p \mapsto\frac{s (p^{k + 1})}{s (p^k)}$ is the same as the sign of $p^{k + 1} (k + 2) - p^{k + 2} (k + 1) - 1 = p^{k +
1} (k + 2 - p (k + 1)) - 1 \leqslant p^{k + 1} (k + 2 - 2 (k + 1)) - 1 = -
p^{k + 1} (k - 1) < 0$.
\item[$7.$ ] The equality 
\[ \frac{s (p^a)}{s (p^b)} = \underset{k = b}{\overset{a - 1}{\prod}} \frac{s(p^{k + 1})}{s (p^k)} \]                                                                                                                                                                                                                                                                                                                                                                                                                                                                                                                                                                                                                                                                                                                                                                                                                                                                                                             along with the monotonicity of each fraction of the product yields the desired result.
\item[$8.$] The function $a \mapsto \dfrac{s (p^{a + 1})}{s (p^a)}$ is decreasing, and we
  have
  \[ \frac{s (p^{a + 1})}{s (p^a)} - s (q) < \frac{s (p^2)}{s (p)} - 1 -
     \frac{1}{q} = \frac{1}{p (p + 1)} - \frac{1}{q} \]
 Hence
  \[ \frac{s (p^{a + 1})}{s (p^a)} < s (q) \leqslant s (q^b) \]
\item[$9.$]Let $n = \underset{p | n}{\prod} p^{a_p}$ and $m = \underset{p | m
}{\prod} p^{b_p}$. We have
\[ m\, n =  \underset{p \nmid m}{\underset{p | n }{\prod}}
   p^{a_p}  \underset{}{\underset{p | m }{\underset{p
   \nmid n}{\prod}} p^{b_p}} \underset{p | 
   m}{\underset{p | n }{\prod}} p^{a_p + b_p} \]
yielding
\[ s (m\, n) = \underset{p \nmid m}{\underset{p | n }{\prod}} s
   (p^{a_p}) \underset{}{\underset{p | m
   }{\underset{p \nmid n}{\prod}} s (p^{b_p})}
   \underset{p |  m}{\underset{p | n }{\prod}} s (p^{a_p +
   b_p})  \]
If $p$ is a prime number, and $a, b$ two positive integers, then
\[ s (p^{a + b}) - s (p^a) s (p^b) = - \frac{(p^a - 1) (p^b - 1)}{p^{a + b -
   1} (p - 1)^2} < 0 \]
Now we have
\[ s (m\, n) \leqslant  \underset{p \nmid m}{\underset{p | n
   }{\prod}} s (p^{a_p})  \underset{}{\underset{p | m
   }{\underset{p \nmid n}{\prod}} s (p^{b_p})}
   \underset{p |  m}{\underset{p | n }{\prod}}( s (p^{a_p})
   s (p^{b_p}) ) = s (m) s (n) \]
\end{enumerate}
\end{proof}
\ \\
\ \\
\begin{lemma}\label{sbound}
For all positive integers $n \geqslant 2$, we have $s(n) < f(N_{\omega(n)})$
\end{lemma}
\ \\
\begin{proof}
Let $n = \underset{i = 1}{\overset{k}{\prod}} q^{a_i}_i$  where $q_1, q_2, q_3, ...$ is a finite increasing sequence of prime numbers. Therefore, for all $i$, $p_i \leqslant q_i$. Set $m = \underset{i = 1}{\overset{k}{\prod}} p^{a_i}_i $. Clearly, $m \leqslant  n$ and 
$$
\frac{s(p^{a_i}_i)}{s(q^{a_i}_i)} \geqslant 1
$$
yielding 
$$
s(n) \leqslant s(m) < f(m) = f(N_{\omega(n)})
$$
\end{proof}
\ \\
\begin{theorem}{(Multiplicity Permutation)}\label{MP}
Let $n$ be a positive integer, $p$ and $q$ be two prime divisors of $n$ $(p < q)$, and $a$ and $b$ their multiplicity
respectively.
\\ \\\
Consider the integer $n^{\star}$ obtained by switching $a$ and $b$. If $a < b$ 
$$ n^{\star} < n \quad \text{and} \quad s(n^{\star}) > s (n)$$
and if $a>b$
\[ n^{\star} > n \quad \text{and} \quad  s (n^{\star}) < s (n)\]
\end{theorem}
\begin{proof}
Switching $a$ and $b$ means that $n^{\star} = n \left( \frac{p}{q} \right)^{b
- a}$. Hence, if $a<b$, we have  $n^{\star} < n$ and
\[ \frac{s (n)}{s (n^{\star})} = \frac{s (p^a) s (q^b)}{s (p^b) s (q^a)} <
   \frac{s (p^a) s (p^b)}{s (p^b) s (p^a)} = 1\]
Indeed, since $a < b$, we know from lemma \ref{AF} that $\dfrac{s (q^b)}{s (q^a)}$ decreases as $q$ increases, and since $q > p$, we have\[ \frac{s (q^b)}{s (q^a)} < \frac{s (p^b)}{s (p^a)} \]
\ \\
The proof for the case $a>b$ follows the same logic.
\end{proof}
\ \\ \\
Let $n = \underset{i = 1}{\overset{k}{\prod}} q^{a_i}_i $, and $b_1, b_2, \ldots, b_k$  be a reordering of $a_1, a_2, \ldots, a_k$ such that $b_i
\geqslant b_{i + 1}$.  
Define the functions $A$  and $H$ as follows:
\begin{equation}
\label{adef}
A \left( \underset{i \leqslant k}{\prod} q^{a_i}_i \right) = \underset{i
   \leqslant k}{\prod} q^{b_i}_i 
\end{equation}
\begin{equation}
\label{hrdef}
H \left( \underset{i \leqslant k}{\prod} q^{a_i}_i \right) = \underset{i
   \leqslant k}{\prod} p^{b_i}_i 
\end{equation}
\ \\
We claim that $ A(n) \leqslant n$ and $s(n) \leqslant s(A(n))$:
let $n_1$ be the
integer obtained by switching $a_1$ and $b_1$. Using theorem \ref{MP}, we have $n_1 \leqslant  n$ and $s (n_1)
\geqslant  s (n) $.
Continue by switching $b_2$ and the multiplicity of $p_2$ in the prime factorization of $n_1$. The result is an integer $n_2$ such that $n_2
\leqslant n_1 \leqslant  n$ and $s (n_2) \geqslant s (n_1) \geqslant s (n)$.
\\
Repeat the process a total of $k-1$ times, and you obtain a positive integer
$n_{k-1}$ such that
\[ n_{k-1} = \underset{i \leqslant k}{\prod} q^{b_i}_i = A(n) \qquad \text{with\,\,} b_i \geqslant b_{i + 1}  \]
\[ n_{k-1} \leqslant n \quad \text{and} \quad s (n_{k-1}) \geqslant s (n) \]
\ \\
This proves that $A(n) \leqslant n$ and $s(A(n)) \geqslant s(n)$.  The same inequalities hold for the function $H$.
The function $H$ replaces all the prime divisors by the first prime divisors, and reorganize the multiplicities in a decreasing order. It is easy to check that $H(n) \leqslant n$ and $s(H(n)) \geqslant s(n)$. Indeed, note that $H(n) \leqslant A(n) \leqslant n$ and $s(H(n)) \geqslant s(A(n)) \geqslant s(n)$.\\
Let $n \geqslant 3$.  If $H(n)$ satisfies Robin's inequality, so does n. Indeed,\\
\[ s (n) \leqslant s (H (n)) < e^{\gamma} \log \log H (n) < e^{\gamma} \log
   \log n \]
thus proving: \\
\begin{theorem}
If Robin's inequality is valid for all Hardy-Ramanujan numbers greater than $5040$, then it is valid for all positive integers greater than $5040$.
\end{theorem}
\ \\
Note: compare this result to proposition 5.1 in \cite{cho}: {\it "if Robin’s inequality holds for all Hardy-Ramanujan integers $5041 \leqslant n \leqslant x$, then it holds for all integers $5041 \leqslant n \leqslant x$"}.
\section{Numbers satisfying Robin's inequality}\label{subsets}
Some categories of numbers satisfy Robin's inequality. Most of them  have already been identified in previous works\cite{cho}, and we report these results below. Note that the proofs are not the original ones. Instead, we used simple methods to provide shorter and simpler proofs. To do so, we will need the following inequalities (\cite{robin}, th. 2).
\ \\
\begin{theorem}\label{robinsig}
For all $n \geqslant 3$, 
$$s (n) \leqslant e^{\gamma} \log \log n + \frac{0.6483}{\log \log n}$$
\end{theorem}
and \cite{rosser}(th. 15)\\
\begin{theorem}\label{fnrosser}
  For all $n \geqslant 3$,
  \[ f (n) \leqslant e^{\gamma} \log \log n + \frac{2.51}{\log \log n} \]
\end{theorem}
\ \\
We will start with primorials. We will then investigate odd positive integers, and integers $n$ such that $\omega(n) \leqslant 4$ (all known integers not satisfying Robin's inequality are such that $\omega(n) \leqslant $4), and we will conclude with square-free  and square-full integers.
\subsection{Primorial numbers\label{prisec}}
\begin{theorem}
All primorials $N_k$ ($k \geqslant 4$) satisfy Robin's inequality:
$$s(N_k)=\underset{i = 1}{\overset{k}{\prod}} \left( 1 + \frac{1}{p_i} \right) <
e^{\gamma} \log \log N_k$$
The following inequality is a bit sharper (for $k \geqslant 2$):
$$\underset{i = 1}{\overset{k}{\prod}} \left( 1 + \frac{1}{p_i} \right)  \leqslant\frac{3}{4} \left(
  e^{\gamma} \log \log N_k + \frac{2.51}{\log \log N_k} \right)$$
\end{theorem}
\ \\
\begin{proof}
Let $k \geqslant 2$ be an integer. We have
\[ s (N_k) = s(2) s\!\!\left( \frac{N_k}{2} \right)
   \leqslant s (2)\,f\!\!\left( \frac{N_k}{2} \right) =
   \frac{s (2)}{f (2)} f (N_k) = \frac{3}{4} f (N_k) \]
\ \\
Using theorem \ref{fnrosser}, we have
\begin{eqnarray*}
  s (N_k) - e^{\gamma} \log \log N_k & \leqslant & \frac{3}{4} \left(
  e^{\gamma} \log \log N_k + \frac{2.51}{\log \log N_k} \right) - e^{\gamma}
  \log \log N_k\\
  & \leqslant & - 0.25 e^{\gamma} \log \log N_k + \frac{1.8825}{\log \log
  N_k}\\
  & < & 0
\end{eqnarray*}
if $k \geqslant 6$. For $k \leqslant 5$, only the numbers $N_1=2, N_2=6$ and $N_3=30$ do not satisfy Robin's inequality.
\end{proof}
\subsection{Odd integers}\label{oddsection}
The following theorem can be found in \cite{cho}, but the proof is not the original one. A very simple proof is given instead.
\\
\begin{theorem}\label{odd}
Any odd positive integer $n$ distinct from $3, 5$ and $9$ satisfies Robin's inequality.
\end{theorem}
\ \\
\begin{proof}
Let $n \geqslant 3$ be a positive odd integer. Using theorem \ref{robinsig}, we have
\[ s (n) = \frac{s (2 n)}{s (2)} < \frac{2}{3} \left( e^{\gamma} \log \log (2
   n) + \frac{0.6483}{\log \log (2 n)} \right) < e^\gamma \log \log n\]  \ \\
for all $n \geqslant 17$. Indeed, set $$g (x) = \frac{2}{3} \left( e^{\gamma} \log \log (2 x) + \frac{0.6483}{\log \log (2 x)} \right) - e^{\gamma} \log\log x$$
We have,  for $x \geqslant 2$,
$$ g' (x) < - \frac{1.23 (\log \log (2 x))^2 + (0.43 + 0.6 (\log \log  (2 x))^2) \log x}{x (\log x) (\log (2 x)) (\log \log (2 x))^2} < 0 $$
\\
When $x \geqslant 17, g (x) \leqslant g (17)<0$. Therefore, all odd integers greater than $15$ satisfy Robin's inequality. For odd integers up to $15$, only $3, 5$ and $9$ do not satisfy Robin's inequality.
\end{proof}
\subsection{Integers $n$ such that $\omega(n)\leq4$}
\begin{theorem}
\label{omega4}
Let $n$ such that  $\omega(n)\leq 4$. The exceptions to Robin's inequality such that $\omega(n) = k$ form the sets ${\mathcal C}_k$ where
\begin{eqnarray*}
{\mathcal C}_1 & = & \left\{3,4,5,8,9,16\right\}\\
{\mathcal C}_2 & = & \left\{6, 10, 12, 18, 20, 24, 36, 48, 72\right\}\\
{\mathcal C}_3 & = & \left\{30, 60, 84, 120, 180, 240, 360, 720\right\} \\
{\mathcal C}_4 & = & \left\{840, 2520, 5040	\right\}
\end{eqnarray*}
\end{theorem}
\ \\
\begin{proof}
Numbers  $n$  under $5041$ that do not satisfy Robin's inequality are well known, and the number of prime divisors $\omega(n)$  of these counterexamples does not exceed $4$. We will now show that if $\omega(n) \leqslant 4$, the only counterexamples are the elements of ${\mathcal C}$.
\\ \\
Let $n$ be an integer such that $\omega(n)\leqslant 4$. We have (lemma \ref{sbound}) using the monotonicity of $k \mapsto f(N_k)$  $$s(n) < f(N_{\omega(n)}) \leqslant f(N_4)=4.375$$
Therefore, if $n>116144$,  $\log\log n \geqslant e^{-\gamma} 4.375$, $$ e^{\gamma} \log \log n \geqslant 4.375 >s(n)$$ and Robin's inequality is satisfied if $n > 116144$ and $\omega(n) \leqslant 4$.
Numerical computations confirm that there are no counterexamples to Robin's inequality in the range  $5041 < n \leqslant 116144$.
\end{proof}
\subsection{Square-free integers}
A positive integer $n$ is called {\it{square-free}} if for every prime number
$p$, $p^2$ is not a factor of $n$. Hence $n$ has the form
\[ n = \underset{i \leqslant k}{\prod} q_i  \]
The following two theorems can be found in \cite{cho}, but not the proofs (we provide simpler proofs).\\
\begin{theorem}
Any square-free positive integer distinct from $2,3,5,6,10,30$ satisfy Robin's inequality.
\end{theorem}
\ \\
\begin{proof}
If $n$ is square-free and even, the prime number $2$ has multiplicity $a_1=1$. This implies                                                                                                                                                                                                                                                                                                                                                                                                                                                                                         
  \[ \frac{s (n)}{s (2 n)} = \frac{s (2)}{s (4)} = \frac{6}{7} \]
and using theorem \ref{robinsig}
\[ s (n) < \frac{6}{7} \left( e^{\gamma} \log \log (2 n) + \frac{0.6483}{\log
   \log (2 n)} \right) < e^{\gamma} \log \log n \]
if $n > 418$. If $n \leqslant 418$, numerical computations show that only the
square-free numbers $2, 6, 10, 30$ do not satisfy Robin's inequality.
\\
If $n$ is odd, see theorem \ref{odd}.
\end{proof}
\subsection{Square-full integers}
Let $n$ be an integer.  If for every prime divisor $p$ of $n$, we have $p^2|n$, the integer $n$ is said to be {\it square-full}.
\\
This result can also be found in \cite{cho}. The proof we give here is shorter and simpler.\\
\begin{theorem}\label{square-full}
The only square-full integers not satisfying Robin's inequality are 4, 8, 9, 16 and 36.
\end{theorem}
\ \\
\begin{proof} 
\\
Case 1: $\omega(n)\leqslant 4$. From theorem \ref{omega4}, the only square-full counterexamples are $4,8,9,16,36$.\\ \\
Case 2: $\omega(n) \geqslant 5$. Let $n=  \underset{i =1}{\overset{k}{\prod}} q^{a_i}_i$. Since $n$ is square-full, we have $n\geqslant N_k^2$ and $s(n)<f(N_k)$ (lemma \ref{sbound}). We can now write, using theorem \ref{fnrosser},
$$
s(n)-e^\gamma \log\log n < f(N_{k}) - e^\gamma \log\log N_k^2 
< \frac{2.51}{\log \log N_k} - 1.2345  < -0.0083
$$
\end{proof}
\section{Proof of theorem \ref{halfpi}}

\begin{lemma}\label{jk}
Let $ M (k) = e^{e^{- \gamma} f (N_k)} - \log N_k$
\begin{itemize}
\item[$1.$]  For all $k \geqslant 39$, we have $M(k) \leqslant  \log N_ {\left\lfloor\frac{k}{3}\right\rfloor}$
\item[$2.$]  For all $k \geqslant 18$, we have $M(k) \leqslant  \log N_ {\left\lfloor\frac{k}{2}\right\rfloor}$
\item[$3.$]  For all $k \geqslant 969672728$, we have $M(k) \leqslant  \log N_ {\left\lfloor\frac{k}{14}\right\rfloor}$
\end{itemize}
\end{lemma}
\ \\
\begin{proof} Let us prove property $1.$\\
We claim that for all $k > 38$, we have $M (3 k + j) \leqslant \log N (k)$ for
all $0 \leqslant j < 3$. Using theorem \ref{fnrosser}, we have
\[ M (k) \leqslant \left( e^{\frac{1.41}{\log \log N_k}} - 1 \right) \log N_k
\] \\
Define
\[ \epsilon_k = e^{\frac{1.41}{\log (3 k \log (3 k))}} - 1 - \frac{k \log
   k}{(3 k + 2) (\log (3 k + 2) + \log \log (3 k + 2))}  \] \\
Using the inequalities \cite{massias}
  \[ k \log k < \log N_k <  k\,(\log k + \log \log  k) \qquad \text{for\  } k \geqslant 13 \]
\ \\
we see that 
$$
\epsilon_k  \geqslant
   e^{\frac{1.41}{\log \log N_{3 k}}} - 1 - \frac{\log N_k}{\log N_{3 k + 2 }}
$$
\ \\
We claim that $\epsilon_k$ is decreasing, and that $\epsilon_k < 0$ for $k>109$. Indeed, rewrite the
term
\[ \frac{k \log k}{(3 k + 2) (\log (3 k + 2) + \log \log (3 k + 2))} =
   \frac{k}{3 k + 2} \times \frac{\log k}{\log (3 k + 2)} \times
   \frac{\log (3 k + 2)}{\log (3 k + 2) + \log \log (3 k + 2)} \]
\ \\
as a product of three increasing functions: to prove they are increasing, you can use the monotonicity of $x
\mapsto \frac{\log x}{x}$ ($x \geqslant e$), and the monotonicity of  $x
\mapsto \frac{\log x}{\log (x + a)} $, where $a$ is a positive real number .\\
\ \\We have $\epsilon_{109} < - 0.0003$ implying for all $0 \leqslant i < 3$
\[ e^{\frac{1.41}{\log \log N_{3 k + i}}} - 1 < e^{\frac{1.41}{\log \log
   N_{3 k}}} - 1 \leqslant \frac{\log N_k}{\log N_{3 k + 2}} \leqslant
   \frac{\log N_k}{\log N_{3 k + i}} \]
yielding
\[ M (3 k + i) \leqslant \left( e^{\frac{1.41}{\log \log N_{3 k + i}}} - 1
   \right) \log N_{3 k + i} \leqslant \log N_k \]
or simply
\[ M (k) \leqslant \log N_{\lfloor \frac{k}{3} \rfloor} \]
\\
for $k \geqslant 109$, and was confirmed for $39 \leqslant k < 109$ with the help of numerical computations.
\ \\ \\
$\bullet$ {\bf Proof of property 2:}
This  result is deduced the same way: consider the function 
$$
\epsilon_k = e^{\frac{1.41}{\log (2 k \log (2 k))}} - 1 - \frac{k \log
   k}{(2 k + 1) (\log (2 k + 1) + \log \log (2 k + 1))}  \\
$$
We have 
$$
\epsilon_k  \geqslant
   e^{\frac{1.41}{\log \log N_{2 k}}} - 1 - \frac{\log N_k}{\log N_{2 k + 1 }}
$$
and $\epsilon_k$ is a decreasing function with $\epsilon_{28} < -0.003$ yielding
$$ 
 M (k) \leqslant \log N_{\lfloor \frac{k}{2} \rfloor} 
$$
for $k > 27$, and was confirmed for $18 \leqslant k < 28$ with the help of numerical computations.\\ \\
$\bullet$ {\bf Proof of property 3:} we proceed the same way. Consider the function 
$$
\epsilon_k = e^{\frac{1.41}{\log (14 k \log (14 k))}} - 1 - \frac{k \log
   k}{(14 k + 13) (\log (14 k + 13) + \log \log (14 k + 13))}  \\
$$
We have 
$$
\epsilon_k  \geqslant
   e^{\frac{1.41}{\log \log N_{14 k}}} - 1 - \frac{\log N_k}{\log N_{14 k + 13 }}
$$
and $\epsilon_k$ is a decreasing function with $\epsilon_{969672728} < -0.001$ yielding
$$ 
 M (k) \leqslant \log N_{\lfloor \frac{k}{14} \rfloor} 
$$
for $k \geqslant 969672728$.
\end{proof}
\ \\ \\
{\bf Proof of theorem \ref{halfpi}} \\
We will need the function $H$ defined in (\ref{hrdef}).
  \[ H (n) \leqslant n \quad \text{and} \quad s (H (n)) \geqslant s (n) \]
\ \\
The integer $H(n)$ is a Hardy-Ramanujan number. In addition, if $H(n)$ satisfies Robin's inequality, so does $n$. 
Indeed,
  \[ s (n) \leqslant s (H (n)) < e^{\gamma} \log \log H(n) \leqslant
     e^{\gamma} \log \log n \]
\ \\
 Note that the value of $\# \left\{ i \leqslant \omega(n)\: ; \: a_i \neq 1 \right\}$ is
  the same if we replace $n$ with $H (n)$, and that  $\omega(H(n)) = \omega (n)$. So without
  loss of generality, we will prove the theorem for Hardy-Ramanujan numbers only.
  \ \\ \\
Let $n$ be a Hardy-Ramanujan number. Define $j (n) : = \min \{ 1 \leqslant j
\leqslant \omega (n) ; M (\omega (n)) \leqslant \log N_j \}$ and $i (n) : = \,
\# \{ i \leqslant \omega (n) ; a_i \neq 1 \}$. A sufficient condition for $n$ to satisfy Robin's inequality is $i (n)\geqslant j (n)$.
Since $n$ is a Hardy-Ramanujan number, we have $a_i \geqslant 2$ for $i
\leqslant i (n)$, and $a_i = 1$ otherwise, implying that $n \geqslant N_{i(n)} N_{\omega (n)}$ and yielding\\
\[ s (n) - e^{\gamma} \log \log n < f (N_{\omega (n)}) - e^{\gamma} \log \log
   (N_{i (n)} N_{\omega (n)}) < 0 \]\\
if $i (n) \geqslant j (n)$. Hence, $n$ satisfies Robin's inequality.
\ \\ \\
If $\omega (n) \geqslant 18$, and if $\# \{ i \leqslant \omega (n) ; a_i \neq
1 \} \geqslant \frac{\omega (n)}{2}$, $n$ satisfies Robin's inequality.
Indeed, if $\omega (n) \geqslant 18$, lemma \ref{jk} tells us that $j (n) \leqslant
\frac{\omega (n)}{2}$. Using the assumption $\# \{ i \leqslant \omega (n) ;
a_i \neq 1 \} \geqslant \frac{\omega (n)}{2}$ yields $i (n) \geqslant j (n)$.
\ \\ \\
Proceed the same way with the two cases: $\omega (n) \geqslant 39$ and $\omega (n)
\geqslant 969672728$
\hfill$\square$
\ \\
\section{Proof of theorem \ref{domination}}
According to lemma \ref{sbound}, we have $s (n) < f (N_{\omega (n)})$ and 
$$
 n = \underset{j = 1}{\overset{\omega(n)}{\prod}} q^{a_j}_j \geqslant q_i^{a_i-1} \left(q_i \prod_{j \neq i}q_j^{a_j} \right) \geqslant q_i^{a_i-1} \left(\underset{j\neq i}{p_i \prod p_j}
\right) = q_i^{a_i-1} N_{\omega(n)}
$$
implying \\
\[ s (n) - e^{\gamma} \log \log n < f (N_{\omega (n)}) - e^{\gamma} \log \log
   (q_i^{a_i - 1} N_{\omega (n)}) \ \]
\\
If $a_i \geqslant M_{\omega(n)} (q_i)$  with
\[ M_k (q) := 1 + \frac{e^{e^{- \gamma} f (N_{k})} - \log N_{k}}{\log q}  
\]
then
$$
 f (N_{\omega (n)}) - e^{\gamma} \log \log
   (q_i^{a_i - 1} N_{\omega (n)}) \leqslant 0
$$ and therefore $$ s (n) < e^{\gamma} \log \log n $$
\section{Proof of theorem \ref{counter}}
\begin{enumerate}
\item[Proof of $1.$] The superabundance of the number $c$ is proved in \cite{ak1}.
\item[Proof of  $2.$]In \cite{briggs}, numerical computations have confirmed that $c > 10^{10^{10}}$.\\
Using lemma \ref{sbound}, we have
\[ e^{\gamma} \log \log c \leqslant s (c) < f (N_{\omega (c)}) < e^{\gamma}
   \left( \log p_{\omega (c)} + \frac{1}{ \log p_{\omega (c)}} \right) \]
yielding
$$
\log p_{\omega(c)}>\frac{1}{2} \left(\log \log c +\sqrt{\log \log c-4} \right)
$$
\\
Since $\log \log c  >  10 \log 10 + \log \log 10 > 23.85988$, we obtain
\[ \log p_{\omega (c)} > 23.81789 \]
yielding
\[ \omega (c) \geqslant 969 \, 672 \, 728 \]\\
\item[Proof of $3.$]{\it $\star$ 	Part 1: $p_\omega(c) <\log c$}\\ \\
 By definition,  the counterexample $c$ does not satisfy Robin's inequality:
\[ s (c) \geqslant e^{\gamma} \log \log c \]
and since $c > 5040$ is the least one, the integer $c'=\dfrac{c}{p_{\omega(c)}}$ 
satisfy Robin's inequality,  because $\omega(c')=\omega(c)-1 > 4$.
 Therefore, we have
\[ s \left( \frac{c}{p_{\omega(c)}} \right) < e^{\gamma} \log \log \frac{c}{p_{\omega(c)}}\]
yielding
\begin{equation}\label{eq2}
  \frac{s (c)}{s \left( \dfrac{c}{p_{\omega(c)}} \right)} >  \dfrac{\log \log c}{\log
  \log \frac{c}{p_{\omega(c)}}}
\end{equation}
Now use $\dfrac{s (c)}{s \left( \frac{c}{p_{\omega(c)}} \right)} = s (p_{\omega(c)}) = 1 +
\dfrac{1}{p_{\omega(c)}}$, and using  inequality $\log(1-x) < x \, (x \neq 0)$,	we establish that
\begin{equation}
  \log \log \frac{c}{p_{\omega(c)}}  = \log \log c + \log \left( 1 - \frac{\log
  p_{\omega(c)}}{\log c} \right) < \log \log c - \frac{\log p_{\omega(c)}}{\log c}
\end{equation}
\\ \\
Inequality (\ref{eq2}) becomes
$$
  \frac{1}{p_{\omega(c)}}  > \frac{\log \log c}{\log \log \frac{c}{p_{\omega(c)}}} - 1  >  \frac{\log \log c}{\log \log c - \frac{\log p_{\omega(c)}}{\log c}} - 1   = \frac{\log p_{_{\omega(c)}}}{\log c \log \log c - \log p_{\omega(c)}}
$$
and
$$
 p_{\omega(c)} \log p_{\omega(c)} < \log c \log \log c - \log p_{\omega(c)} < \log c \log \log c
$$
\ \\
implying $p_{\omega(c)} < \log c.$
\ \\ \\ \\
{\it $\star$  Part 2: $\log c < p_{\omega (c)} e^{\frac{1}{\log p_{\omega  (c)}}}$}
\\ \\
The counterexample inequality $s (c) \geqslant e^{\gamma} \log \log c$ yields
\[ e^{\gamma} \log \log c \leqslant s (c) < f (N_{\omega(c)}) < e^{\gamma} \left( \log
   p_{\omega(c)} + \frac{1}{\log p_{\omega(c)}} \right) \]
implying $\log c < p_{\omega (c)} e^{\frac{1}{\log p_{\omega (c)}}}$.\\ \\
\item[Proof of $4.$] See theorem \ref{halfpi}.
\\
\item[Proof of $5.$] The proof of the inequality $p_i^{a_i} < 2^{a_1+2}$  can be found in \cite{erdos} (lemma 1). The other inequality
$$ 
p_i ^{c_i}< p_i e^{M(\omega(c))}
$$
is deduced from theorem \ref{domination}.
\end{enumerate}
\section{Sharper bounds}\label{sharp}
We can use some of the ideas of this article to improve upper bounds for the
arithmetic functions $s$ and $f$.
\subsection{Primorials}
 For instance, when dealing with primorials $N_k$, we can write
\[ s (N_k) = \frac{s (2)}{s (4)} s (2 N_k) \]
and then use an upper bound for $s (2 N_k)$ (see theorem \ref{robinsig}) to sharpen Robin's bound for $s
(N_k)$. Or we can use another method :\ \\
\begin{theorem}
For all $i \geqslant 2$, for all $n \geqslant i$,
\[ \underset{k = 1}{\overset{n}{\prod}} \left( 1 + \frac{1}{p_k} \right) <
   \alpha_i \left( e^{\gamma} \log \log N_n + \frac{2.51}{\log \log N_n}
   \right) \]
with
\[ \alpha_i = \underset{j = 1}{\overset{i}{\prod}} \left( 1 - \frac{1}{p^2_j}
   \right) \]
\end{theorem}
Ex: for all $n \geqslant 4$,  \[ \underset{k = 1}{\overset{n}{\prod}} \left( 1 + \frac{1}{p_k} \right) <
   0.627 \left( e^{\gamma} \log \log N_n + \frac{2.51}{\log \log N_n}
   \right)   <  e^{\gamma} \log \log N_n \]
\begin{proof}
Let $i \geqslant 2, n \geqslant i$. We  have
\[ s (N_n) = s (N_i)  s\left( \frac{N_n}{N_i} \right) <  s (N_i) f \left(
   \frac{N_n}{N_i} \right) = \frac{s (N_i)}{f (N_i)} f (N_n) = \alpha_i f
   (N_n) \]
Use theorem \ref{fnrosser} to conclude.
\end{proof}
\subsection{Odd numbers}
\begin{lemma}
If $n \geqslant 17$ is odd,
\[ \sigma (2 n) < 2 e^{\gamma} n \log \log (2 n) \]
\end{lemma}
\ \\
\begin{proof}
Write 
\[ s (2 n) = \frac{s (2)}{s (4)} s (4 n) \leqslant \frac{6}{7}  \left(
   e^{\gamma} \log \log (4 n) + \frac{0.6483}{\log \log (4 n)} \right) <
   e^{\gamma} \log \log (2 n) \]
To prove the last inequality, define
\[ g (x) = \frac{6}{7}  \left( e^{\gamma} \log \log (4 x) +
   \frac{0.6483}{\log \log (4 x)} \right) -  e^{\gamma} \log \log (2 x) \]
We have
\[ g' (x) < - \frac{0.38 + 1.41 (\log \log (4 x))^2 + (0.55 + 0.25 (\log \log
   (4 x))^2) \log x}{x \log (2 x) \log (4 x) (\log \log (4 x))^2} < 0 \]
\ \\
and $g (210) < 0 < g (209)$. The inequality $s (2 n) < e^{\gamma} \log \log (2
n)$ is valid for $n > 209$, and numerical computations show that this is also
true when $17 \leqslant n \leqslant 209$.
\end{proof}
\ \\
\begin{corollary}
If $n \geqslant 17$ is odd,
$$\sigma (n) < \frac{2}{3} e^{\gamma} n \log \log (2 n) $$
\end{corollary}	
\ \\
The following result is an improvement of theorem 1.2  in \cite{cho}.
\ \\ \\
\begin{proof}
Use the previous result to get
\[ s (n) = \frac{s (2 n)}{s (2)} < \frac{2}{3} e^{\gamma} \log \log (2 n) \]
\end{proof}
\ \\
The following result is an improvement of theorem 2.1 in \cite{cho}.\\
\begin{theorem}
If $n \geqslant 3$ is odd,
\[ \frac{n}{\varphi(n)} \leqslant \frac{1}{2}  \left( e^{\gamma} \log \log (2 n) + \frac{2.51}{\log
   \log (2 n)} \right) <  e^{\gamma} \log \log n \]
\end {theorem}
\begin{proof}
Write $$f(n)=\frac{f(2n)}{f(2)}$$ ans use theorem \ref{fnrosser}:
$$f(n) \leqslant \frac{1}{2}  \left( e^{\gamma} \log \log (2 n) + \frac{2.51}{\log
   \log (2 n)} \right) $$
\end{proof}
\subsection{The general case}
In previous sections, we have mentioned superabundant numbers and pointed out
some of their properties, thus confirming that all superabundant numbers are
Hardy-Ramanujan numbers.
\ \\ \\
Robin's inequality can be enhanced for non superabundant numbers. Indeed, if
$n$ is not a superabundant number, there exists $m < n$ such that $s (m)
\geqslant s (n)$. Let $I (n) = \{ m < n ; s (n) \leqslant s (m) \}$. The set
$I (n)$ is not empty, and we can define $B (n) = \min I (n)$. We have
\[ B (n) < n \qquad \text{and} \qquad s (n) \leqslant s (B (n) \]
yielding
\[ s (n) \leqslant s (B (n)) < e^{\gamma} \log \log B (n) \]
if $B (n)$ satisfies Robin's inequality. Hence, the upper bound is better
because $B (n) < n$.
\ \\ \\
But this method to sharpen Robin's bound does not say much about $B (n)$. In
order to find an upper bound with exploitable informations, we can use the
function $H$ (see (\ref{hrdef})). This method concerns non Hardy-Ramanujan
numbers, whereas the previous method concerns the larger set of non
superabundant numbers.
\ \\
In (\ref{hrdef}), we defined the function $H$ which transforms any number into a
Hardy-Ramanujan number,  leaves the multiplicities unchanged (but not teh way they are ordered) and is such 
that $H (n) \leqslant n$ and $s (n) \leqslant s (H(n))$. If $H (n)$ satisfies
Robin's inequality, then
\[ s (n) \leqslant s (H (n)) < e^{\gamma} \log \log H(n) \]
The property $H(n) \leqslant n$ improves Robin's inequality, and unlike the function $B$, there is an explicit method to determine its
value for all positive integers.

\end{document}